\documentclass[twoside, a4paper, nofootinbib,10pt]{article}
\usepackage{amssymb}
\usepackage[mathscr]{eucal}
\usepackage{graphicx}
\usepackage{amsmath}
\usepackage{amsthm}
\def \ni{\noindent}

\newcommand{\be}{\begin{equation}}
\newcommand{\ee}{\end{equation}}
\newcommand{\ben}{\begin{equation*}}
\newcommand{\een}{\end{equation*}}
\newcommand{\bes}{\begin{eqnarray}}
\newcommand{\ees}{\end{eqnarray}}
\newcommand{\besn}{\begin{eqnarray*}}
\newcommand{\eesn}{\end{eqnarray*}}
\newcommand{\txt}{\textrm}
\newtheorem{theorem}{Theorem}
\newtheorem*{lemma}{Lemma}
\newtheorem{definition}{Definition}

\newtheorem*{theorem*}{Theorem}

\title{On the Holonomic Equivalence of Two Curves}
\author{Tamer Tlas} 
\date{}

\begin{document}
\maketitle

\abstract{\textit{ Given a principal $G$-bundle $P \to M$ and two $C^1$ curves in $M$ with coinciding endpoints, we say that the two curves are holonomically equivalent if the parallel transport along them is identical for any smooth connection on $P$. The main result in this paper is that if $G$ is semi-simple, then the two curves are holonomically equivalent if and only if there is a thin, i.e. of rank at most one, $C^1$ homotopy linking them. Additionally, it is also demonstrated that this is equivalent to the factorizability through a tree of the loop formed from the two curves and to the reducibility of a certain transfinite word associated to this loop. The curves are not assumed to be regular. \vspace{1cm}
}}

Let $M$ be a smooth manifold, $G$ a Lie group and $P \to M$ a smooth principal $G$-bundle. Let $\mathcal{A}$ be the space of smooth connections on $P$ and $\mathcal{G}$ the space of smooth gauge transformations. It is a well-known fact that $\mathcal{A}/\mathcal{G}$ is one of the richest and most interesting objects to study, with applications ranging from knot invariants to high energy physics. Unfortunately, although $\mathcal{A}$ is an affine space, the quotient is, due to the nonlinearity of the action of $\mathcal{G}$, not even an infinite dimensional manifold. It is natural to try to find a description of $\mathcal{A}$ on which the action of $\mathcal{G}$ would be more manageable.

Fix a point $o \in M$ and suppose $\gamma$ is a loop based at $o$, i.e. it is a, sufficiently differentiable, map from $I = [0,1] \to M$, and which satisfies $\gamma(0) = \gamma(1) = o$. Any element $A \in \mathcal{A}$ associates to this loop a bijection from the fiber over $o$ to itself, and if one fixes some point in this fiber, this bijection can be identified with an element of $G$. This element shall be denoted by $U(\gamma, A)$ and is called the holonomy of $A$ around the loop $\gamma$. The key point is that $\mathcal{G}$ acts very transparently on the holonomy. The holonomy is simply conjugated, with the conjugating element being the same for all the loops based at $o$. It is thus natural to look for a formulation of gauge theory where the holonomies are taken as elementary objects and the connection, and maybe even the bundle, are reconstructed from them afterwards if needed, and indeed, there is a long history of attempts in this direction dating back to \cite{kobayashi}. 

Let us denote the set of loops based at $o$ by $\mathcal{L}$. It is easy to see that not any function from $\mathcal{L}$ to $G$ is a holonomy of some connection. The reason is that if one defines the product of two loops in the usual way\footnote{
\ben
\gamma_1 \cdot \gamma_2 (t) = \left\{ \begin{array}{ll}
\gamma_1(2t) & t \in [0, \frac{1}{2}] \\
 \gamma_2(2t - 1) & t \in [\frac{1}{2}, 1] 
 \end{array} \right.
\een
 } then $U(\gamma_1\cdot \gamma_2, A) = U(\gamma_2, A) U(\gamma_1, A)$. This equation means that, had $\mathcal{L}$ been a group, an element of $\mathcal{A}$ would simply be a homomorphism from this group to $G$, and an element of $\mathcal{A}/\mathcal{G}$ would just be the conjugacy class of this homomorphism. But, the set of loops is not a group, as the product of two loops is neither associative nor does it have an identity. The way to solve this problem is of course clear: one should take an appropriate quotient of the set of loops as is done similarly in elementary homotopy theory. It is evident, however, that taking the quotient of $\mathcal{L}$ by the standard homotopy is inconsistent, as in general, two homotopic loops will have different holonomies. Thus, we are faced with the problem of finding an equivalence relation $\sim$ on $\mathcal{L}$ such that $\{ \mathcal{L}, \cdot \}/\sim$ is a group and moreover \ben \gamma_1 \sim \gamma_2 \implies U(\gamma_1, A) = U(\gamma_2, A) \quad , \quad  \forall A \in \mathcal{A}.\een

In fact, it is natural to require the implication to go in the opposite direction as well, since otherwise it would mean that a connection is not a homomorphism from $\{\mathcal{L}, \cdot \}/\sim$ to $G$ but from a further quotient of it.

The above motivates the following

\begin{definition}
Two loops $\gamma_1, \gamma_2 \in \mathcal{L}$ which satisfy $U(\gamma_1, A) = U(\gamma_2, A)$ for all $A \in \mathcal{A}$ are called holonomically equivalent. 
\end{definition}

Let us stress the point that in this definition we assume $P \to M$, and thus $G$, to be fixed.

In view of the discussion above, we are tasked with finding necessary and sufficient conditions for holonomic equivalence. The answer to this problem depends on the class of paths under consideration. In general, finding a sufficient condition is much easier than obtaining a necessary one. Let us discuss the different classes of paths which have been considered in the literature so far as well as the different associated sufficiency and necessity conditions. Chronologically: 

\begin{itemize}
\item[i-] In \cite{kobayashi}, the loops were assumed to be piecewise differentiable. The equivalence relation was that $\gamma_1 \sim \gamma_2$ if $\gamma_1 \cdot \overline{\gamma}_2$ factors through a finite tree, \cite{kobayashi}.\footnote{This is not the formulation that was used in \cite{kobayashi} but is in fact equivalent to it. The same is true, where applicable, in the other cases.} $\overline{\gamma}$ stands for the loop $\gamma$ traversed in the opposite direction.  Necessity of this relation for holonomic equivalence was not considered, only sufficiency.
\item[ii-] In \cite{chen}, the loops were assumed to be piecewise $C^1$ as well as regular, i.e. that $\dot{\gamma}(t) \neq 0$ for all $t \in I$. The equivalence relation was the same as in [i]. It was shown there that two curves are equivalent if and only if all the Chen's iterated integrals of the loop $\gamma_1 \cdot \overline{\gamma}_2$ vanish. If $\gamma$ is a loop, then its iterated integrals are

\besn
&&\int_0^1dt \, \dot{\gamma}^\mu(t)\\
&&\int_0^1 dt \int_0^t ds \, \dot{\gamma}^\mu(t) \dot{\gamma}^\nu(s)\\
&&\int_0^1 dt \int_0^t ds \int_0^s dr \, \dot{\gamma}^\mu(t) \dot{\gamma}^\nu(s) \dot{\gamma}^\rho(r)\\
&& \vdots
\eesn

Using standard properties of the iterated integrals, it is not very difficult to show that the above set of expressions vanishes if and only if for any collection of smooth 1-forms $w^1, w^2, w^3 \dots$ (which can even be taken to be matrix-valued), the following integrals vanish

\besn
&&\int_0^1 dt \, w^1_\mu(\gamma(t))  \, \dot{\gamma}^\mu(t) \\
&&\int_0^1 dt \int_0^t ds \, w^1_\mu(\gamma(t)) \, w^2_\nu (\gamma(s)) \,  \dot{\gamma}^\mu(t) \dot{\gamma}^\nu(s)\\
&&\int_0^1 dt \int_0^t ds \int_0^s dr \, w^1_\mu(\gamma(t)) \, w^2_\nu(\gamma(s))\, w^3_\rho(\gamma(r)) \, \dot{\gamma}^\mu(t) \dot{\gamma}^\nu(s) \dot{\gamma}^\rho(r)\\
&&\vdots
\eesn

Recalling that $U(\gamma,A)$ has an expression in terms of the path-ordered exponential of $A = A_\mu dx^\mu$

\ben
I + \int_0^1dt A_\mu (\gamma(t)) \dot{\gamma}^\mu(t) + \int_0^1 dt \int_0^t ds A_\mu(\gamma(t)) A_\nu(\gamma(s)) \dot{\gamma}^\mu(t) \dot{\gamma}^\mu(s) + \dots
\een

it is clear that the fact that all the iterated integrals vanish implies that the holonomy around $\gamma$ is trivial (equal to the identity) for every connection. If we know that the $U(A,\gamma) = I$ for every $A$ \textit{and} for every $G$, then one can show that this implies that all the iterated integrals vanish \cite{hain}. However, if one is only given that $U(A, \gamma) = I$ for every $A$ and a \textit{specific} $G$ then the triviality of the holonomy is an a priori weaker statement, since one can only conclude (by a scaling argument) that certain sums of iterated integrals vanish. In view of the above, the necessity condition that has been proven in \cite{chen} is weaker than what we seek if we are concerned with the holonomic equivalence for curves.
\item[iii-] In \cite{ashtekar}, the loops are assumed to be piecewise analytic and the equivalence relation is the same as in [i]. Both sufficiency and necessity are demonstrated. 
\item[iv-] In \cite{caetano}, the loops are assumed to be $C^\infty$ and have all their derivatives vanish at the endpoints. Two loops are declared to be equivalent if $\gamma_1 \cdot \overline{\gamma}_2$ is homotopic to the constant loop via a thin homotopy.\footnote{See definition \ref{def:thin}. } Only sufficiency is shown.
\item[v-] In \cite{hambly}, the loops are assumed to be of bounded variation and $\gamma_1 \sim \gamma_2$ if $\gamma_1 \cdot \overline{\gamma}_2$ factors through a tree. It is shown that two curves are equivalent if and only if all the Chen's iterated integrals vanish. The same remarks as the ones in [ii] on the relation of this to holonomic equivalence apply here.
\end{itemize}

In this manuscript we demonstrate that if $G$ is a semi-simple Lie group, and $\gamma_1, \gamma_2$ are two $C^1$ loops having vanishing derivative at the endpoints then the following are equivalent:

\begin{itemize}
\item[a-] $\gamma_1$ and $\gamma_2$ are holonomically equivalent.
\item[b-] $\gamma_1 \cdot \overline{\gamma}_2$ factors through a tree.\footnote{See definition \ref{def:tree}.}
\item[c-] There is a thin $C^1$ homotopy between the constant loop and $\gamma_1 \cdot \overline{\gamma}_2$.
\item[d-] A certain transfinite word (to be defined later) associated to $\gamma_1 \cdot \overline{\gamma}_2$ is reducible.
\end{itemize}

A few remarks are in order:

\begin{itemize}
\item In view of the fact that a piecewise $C^1$ curve can be reparametrized to become $C^1$, it follows that the same statements are true (with an appropriate modification of [c]) for piecewise $C^1$ loops. 
\item We required the vanishing of the derivative of the loops at the endpoints to ensure that their product is $C^1$, as this is the most useful class of loops if we are interested in a loop-based description of gauge theory \cite{caetano}. If we do not require this, then $\gamma_1 \cdot \overline{\gamma}_2$ is piecewise $C^1$ and the previous remark applies.
\item We have only stated the result for loops, but an essentially identical result (see the corollary after theorem 3) applies if $\gamma_1$ and $\gamma_2$ are two curves such that thir initial points as well as their final ones coincide.
\end{itemize}

In what follows, we assume that all the maps are $C^1$ unless stated otherwise.

Let us begin with a certain decomposition of curves which is of interest in its own right:

\begin{theorem} \label{theorem:words}
Given a curve $\gamma : I \to M$, there is a collection of mutually disjoint sets $\{ A_n \}_{n=0}^\infty$, such that:
\begin{itemize}
\item[a-] $A_0$ is closed, while $A_n$ is open for $n >0$, and $\bigcup_{n=0}^\infty A_n = I$. 
\item[b-] If $t \in A_n$ then $\gamma^{-1}(\gamma(t)) \subset A_n$, moreover if $n> 0$ then the cardinality of $\gamma^{-1}(\gamma(t))$ is equal to $n$.
\item[c-] $\dot{\gamma}(t) \neq 0$ if $t \in A_n$ for $n>0$, while $\dot{\gamma}(t) =0$ for t in the interior of any connected component of $A_0$.
\item[d-] The $A_n$'s being open are a union of disjoint open intervals. $\gamma$ restricted to any such interval is an embedding. Moreover, any two such embeddings are either disjoint or identical (possibly after a reparametrization and a switch in orientation). 
\end{itemize}
\end{theorem}

\begin{proof}

Let $C$ be the set of all critical points of $\gamma$, and let $C' = \gamma^{-1} (\gamma(C))$. Note that $C'$ is closed and that it contains all the points $t \in I$ such that $\gamma^{-1}(\gamma(t))$ has infinite cardinality. For assume that $t_0$ is such a point, then this means that $\gamma^{-1}(\gamma(t_0))$ contains a limit point $t_1$. It is obvious that $\dot{\gamma}(t_1) = 0$, i.e. $t_1 \in C$. It follows that $t_0 \in \gamma^{-1}(\gamma(t_1)) \subset C'$.

Also, $\dot{\gamma}(t) = 0$ for $t$ in the interior of any connected component of $C'$. To see this, assume that there is $t_0 \in C'$ with $\dot{\gamma}(t_0) \neq 0$ and for some $\delta >0$, $(t_0 - \delta, t_0 + \delta) \subset C'$. It follows that $\gamma$ is an embedding in a neighborhood of $t_0$, $(t_0 - \delta', t_0 + \delta')$. This in particular implies that the 1-dimensional Hausdorff  measure of $\gamma(C')$ is strictly greater than 0,  since $\gamma(C')$ contains $\gamma( (t_0 - \delta', t_0 + \delta'))$, a set diffeomorphic to an interval. However, $\gamma(C') = \gamma(C)$, and the 1-dimensional Hausdorff measure of $\gamma(C)$ is equal to 0 by the strong version of Sard's theorem \cite{sard}. Therefore there are no open intervals contained in $C'$ where $\dot{\gamma}$ does not vanish.

It follows from above that if $t \notin C'$, then $\gamma^{-1}(\gamma(t))$ is a finite set. For any $n > 0$, let $B_n = \{ t \notin C' : \gamma^{-1}(\gamma(t)) \quad \textrm{has cardinality} \quad n \}$. Let $A_n = B_n^\circ$ (the interior) for $n>0$ and let $A_0 = I - \bigcup_{n=1}^\infty A_n$. It is clear that the $A_n$'s are mutually disjoint and that (a) holds. 

Let us prove (c). By construction, we have that $\dot{\gamma}(t) \neq 0$ if $t \in A_n$ for $n>0$. To prove the rest of (c) assume that there is a point $t_0 \in A_0^\circ$ such that $\dot{\gamma}(t_0) \neq 0$. Since $C'$ does not contain intervals on which $\dot{\gamma}$ does not vanish, and since $C'$ is closed, it follows that $A_0 - C'$ contains an open interval. We will be done if we show that $A_0 - C'$ does not contain open intervals. To see this, note that $A_0 - C' = \bigcup_{n=1}^\infty (B_n - B_n^\circ)$ and assume that $t_0 \in B_m$. We need the following claim:

\textit{
If $s \in B_m$ then there is a $\delta>0$ such that $(s-\delta, s + \delta) \cap \Big ( \bigcup_{n = m+1}^\infty B_n \cup C' \Big )= \phi.$
}

Assuming this statement for the moment, note that it implies that if there is an open interval which is contained in $A_0 - C'$, then it is in fact contained in $ \bigcup_{n=1}^m (B_n - B_n^\circ)$. This interval cannot contain points from $B_1$, since in view of the claim above it would contain a subinterval contained in $B_1$, and thus in $B_1^\circ$ which is a contradiction. Similarly, it cannot contain points from $B_2, B_3, \dots, B_m$, for the same reason. Thus such an interval cannot exist and so $A_0 - C'$ contains no open intervals.

Now, to prove the claim above, assume the converse, then there is a sequence of points $\{s_k \}_{k=1}^\infty \subset \bigcup_{n = m+1}^\infty B_n \cup C' $ which converges to $s$. Since $C'$ is closed, we can assume without loss of generality that this sequence is in fact contained in $\bigcup_{n = m+1}^\infty B_n$. It follows that $\gamma^{-1}(\gamma(s_k))$ always has cardinality at least equal to $m+1$. For every $k$ select $m+1$ points from $\gamma^{-1}(\gamma(s_k))$. In this way, we get $m+1$ sequences in $I$. Passing to subsequences we can assume that they all converge to $L_1, \dots L_{m+1}$ respectively. Moreover, these limits must in fact be different from each other, since otherwise it would mean that one of these limits is in $C$ and thus in $C'$, but this would imply that $s \in C'$ since $\gamma(L_1) = \dots = \gamma(L_{m+1}) = \gamma(s)$. Thus $L_1, \dots, L_{m+1}$ are all different from each other. But this means that $\gamma^{-1}(\gamma(s))$ has cardinality at least equal to $m+1$, which contradicts the fact that $s \in B_m$. Thus the claim above is true and the proof of (c) is complete.

It would follow from the definitions of $B_n$ and $A_n$, that (b) holds if we manage to show that if $t_1 \in B_n^\circ$, then every element of $\gamma^{-1}(\gamma(t_1)) \in B_n^\circ$. To see that this is true, let $\{t_1, t_2, \dots, t_n \} = \gamma^{-1}(\gamma(t_1)) $. By construction we have that $\dot{\gamma}(t) \neq 0$ if $t \in B_n$. It follows that there are positive $\delta_1, \delta_2, \dots, \delta_n$ such that $\gamma$ is an embedding when restricted to $(t_i - \delta_i , t_i + \delta_i), i =1, \dots, n$, and $(t_1 - \delta_1 , t_1 + \delta_1) \subset B_n$. We now make the following claim:

\textit{There is a positive $\delta < \delta_1$ such that $\gamma$ restricted to $(t_2 - \delta_2, t_2 + \delta_2)$ is onto the set $\gamma\big ( (t_1 - \delta, t_1 + \delta) \big ). $ } 

To see that this is the case, assume the converse. Then, there is a sequence $\{ s_k \}_{k=1}^\infty \subset (t_1 - \delta_1, t_1 + \delta_1)$ such that $\gamma^{-1}(\gamma(s_k)) \notin (t_2 - \delta_2 , t_2 + \delta_2)$ for all $k$. Since $\{ s_k \}_{k=1}^\infty \subset B_n$, it follows that for every $k$ there is at least one element of $\gamma^{-1}(\gamma(s_k))$ outside the set $\bigcup_{i=1}^n (t_i - \delta_i, t_i + \delta_i)$. This sequence of elements has a subsequence which converges to a point $t_{n+1} \notin \{t_1, \dots, t_n \}$, and continuity of $\gamma$ implies that $\gamma(t_{n+1}) = \gamma(t_1)$. But this would mean that $\gamma^{-1}(\gamma(t_1))$ has cardinality at least equal to $n+1$ which contradicts the fact that $t_1 \in B_n$. Thus the claim above is true.

Using the fact that $\gamma$ restricted to $(t_2 - \delta_2, t_2 + \delta_2)$ is an embedding, we have that the map $\gamma^{-1} \circ \gamma : (t_1 - \delta, t_1 + \delta) \to (t_2 - \delta_2, t_2 + \delta_2)$ is $C^1$, takes $t_1$ to $t_2$ and has a non-vanishing derivative. Thus the image of $(t_1 - \delta, t_1 + \delta)$ contains an open interval containing $t_2$, which in turn means that there is an open interval around $t_2$ which is contained in $B_n$. Thus $t_2 \in B_n^\circ$. Needless to say, the same argument applies to $t_3, t_4, \dots, t_n$, and the proof of (b) is complete. Note that we have proven in fact that if $t_1, t_2 \in A_n$ and $\gamma(t_1) = \gamma(t_2)$, then there are $\delta, \delta' >0$ such that $\gamma( (t_1 - \delta, t_1+ \delta)) = \gamma(( t_2 - \delta', t_2 + \delta'))$

It remains to prove (d). Since $A_n$ is an open subset of $I$, then it is equal to a countable union of disjoint open intervals. Let $(a,b)$ be one such interval. Then $\gamma$ is in fact injective on this interval. To see this, suppose that $t_1, t_2 \in (a,b)$ with $\gamma(t_1) = \gamma(t_2)$. We know that $\gamma$ is injective on a neighborhood of $t_1$. Thus, let $t'_2$ be the largest of the points in $[t_1, t_2]$ such that $\gamma$ is injective on $[t_1, t_2')$. Let $t_1'$ be the point in $[t_1, t_2')$ for which $\gamma(t_1') = \gamma(t_2')$. Note, that $\gamma$ is, by construction, injective on $(t_1', t_2')$. Consider the set $D$ consisting of all the points $t \in (a,b)$ such that $\exists \, s \in [ t'_1 , t'_2 ]$ with $\gamma(t) = \gamma(s)$. This set is clearly closed in the subspace topology, being the preimage of the closed set $\gamma([t'_1, t'_2])$. This set is also open. To see this, note that if $\gamma(s) = \gamma(t)$ and $s \in (t'_1, t'_2)$, then we have an open interval around $t$ whose image under $\gamma$ coincides with the image of an interval around $s$ which may be taken to be a subset of $(t'_1, t'_2)$. Therefore, this interval around $t$ is a subset of $D$. Suppose now that $\gamma(t) = \gamma(t_1') = \gamma(t_2')$, then we know that there is an open interval around $t$ whose image under $\gamma$ coincides with the image of an open interval around $t_1'$ and with the image of an open interval around $t'_2$. Then the image of half of the interval around $t$ (either the one less or equal or the one bigger or equal to $t$) coincides with the image of the half of the interval around $t_1'$ which is contained in $[t_1', t_2']$ while the other half of the interval around $t$ coincides with the image of the half of the interval around $t_2'$ which is contained in $[t_1', t_2']$ (We used injectivity of $\gamma$ on $(t_1', t_2')$ in a crucial way here). Therefore, the image of the interval around $t$ coincides with the image of a subset of $[t_1', t_2']$ and thus this interval around $t$ is contained in $D$.

We thus have that $D$ is both open and closed in $(a,b)$ and thus, being nonempty, is in fact equal to the whole set. Consider now a sequence $\{t_n \}_{n=1}^\infty$ converging to $a$. Then there is a sequence $\{s_n \}_{n=1}^\infty \subset [t_1', t_2']$ whose image under $\gamma$ coincides with the image of $\{t_n \}_{n=1}^\infty$. Passing to a subsequence and taking the limit we get a point in $[t_1', t_2']$ whose image under $\gamma$ is equal to $\gamma(a)$. We thus have that $a \in A_n$, which contradicts (b).

Summarizing the above, we have shown that $\gamma$ restricted to $(a,b)$ is an injective immersion. To see that it is an embedding, note that $\gamma(a), \gamma(b) \in A_0$. We now have two cases:

\begin{itemize}
\item[i-] $\gamma(a) \neq \gamma(b)$. In this case $\gamma$ restricted to $[a,b]$ is a topological embedding (being an injective continuous map from a compact space to a Hausdorff one).
\item[ii-] $\gamma(a) = \gamma(b)$. In this case $\gamma$ restricted to $[a,b]$ is a topological embedding of $S^1$.
\end{itemize}

It follows that in both cases $\gamma$ restricted to $(a,b)$ is a smooth embedding. 

To finish, assume that $t_1 \in (a_1, b_1)$ and $t_2 \in (a_2, b_2)$, with $(a_1, b_1),$ $ (a_2, b_2)$ being two intervals making up $A_n$, with $\gamma(t_1) = \gamma(t_2)$. Consider the set of points in $(a_1, b_1)$ whose image under $\gamma$ is contained in the image of $(a_2, b_2)$. Again, this set is closed (in the subspace topology of $(a_1, b_1)$, being the preimage of $\gamma( [ a_2, b_2])$), as well as open (by the fact above about the existence of intervals with coinciding images). Thus $\gamma((a_1, b_1)) \subset \gamma((a_2, b_2))$, and obtaining the opposite inclusion via a similar argument we arrive at the fact that $\gamma((a_1, b_1)) = \gamma((a_2, b_2))$. We thus have two embeddings of intervals whose images coincide. Since these are 1-dimensional manifolds, the two embeddings are related by a change of parametrization and possibly a switch in orientation. \end{proof}

The above theorem shows that any $C^1$ curve behaves in a rather simple way once $A_0$ is ignored. This removed set is small topologically in the sense that it is sufficient to know the definition of the curve on its complement to be able to extend it uniquely by continuity to the whole interval. Note however, that neither $A_0$ nor its image can be assumed to be small in the measure-theoretic sense. In fact, the Hausdorff dimensions of these two sets may be equal to one.

Using the above theorem we can associate a transfinite word to any $C^1$ curve. Fix an orientation for each one of the disjoint arcs forming the curve. Let $T$ be the set of all the intervals making up $\bigcup_{n=1}^\infty A_n$ with linear order inherited from $\mathbb{R}$. Choose a countable alphabet $A$, one letter for each one of the arcs. The word associated to the curve is just the map taking every element of $T$ either to the letter corresponding to the arc or to its inverse depending on whether the arc is traversed by $\gamma$ along the chosen orientation when restricted to this element of $T$ or oppositely. This set of maps can be multiplied and reduced in essentially the same way as the set of finite words \cite{hawaii}. 

We can now make an important

\begin{definition}
A `whisker' is a curve whose reduced word is trivial (empty).\footnote{The name `whisker' is motivated by the geometry of the curve whose word is $abb^{-1}c$.} A curve `has whiskers' if its word is different from its reduced word.
\end{definition}

It is easy to see that a transfinite word reduces to a trivial one if and only if every finite truncation of it (i.e. the word obtained by keeping only finitely many of the letters) is reducible. Moreover, this is equivalent with the statement that there is a pairing between any letter in the word with its inverse such that if $a$ appears before $b$ in the word then the $a^{-1}$ that is paired with it appears after the corresponding $b^{-1}$.  

Before we state the next theorem, let us formally define thin homotopy and trees:

\begin{definition}
\label{def:thin}
If $\gamma_1$ and $\gamma_2$ are two curves such that $\gamma_1(0) = \gamma_2(0), \gamma_1(1) = \gamma_2(1), \dot{\gamma}_1(0) = \dot{\gamma}_2(0) = \dot{\gamma}_1(1) = \dot{\gamma}_2(1) =0$, then we say that these two curves are thinly homotopic if there is a map $H: I \times I \to M$ such that 
\besn
H(t,0) & = & \gamma_1 (t) \\
H(t,1) & = & \gamma_2 (t)\\
H(0,s) & = & \gamma_1 (0) = \gamma_2(0) \\
H(1, s) & = & \gamma_1 (1) = \gamma_2 (1)\\
\frac{\partial H (t,0)}{\partial s} = \frac{\partial H (t,1)}{\partial s} &=&  \frac{\partial H (0,s)}{\partial t}  = \frac{\partial H (1,s)}{\partial t}   = 0\\
 rank(H) &\leq & 1
\eesn
Intuitively it is a homotopy which `sweeps' zero area and which `stops', `comes to a halt' at the edges.
\end{definition}

\begin{definition}
\label{def:tree}
Let $l^1$ be the space of absolutely convergent sequences with its natural norm. A tree is a compact, path-connected, simply connected subspace of $l^1$ consisting of a closed, totally disconnected set whose elements are called vertices, and a countable collection of straight line segments called edges. All the points in an edge have all coordinates constant but one. Moreover, this special coordinate for one edge is different from that for any other edge. The endpoints of any edge are contained in the set of vertices, the interiors of the edges are disjoint from each other and from the vertices. Note that according to this definition, a tree is not necessarily a $CW$-complex. We shall assume that we have a preferred vertex which will be called the root of the tree. 
\end{definition}

 We are now ready to state our next

\begin{theorem} \label{theorem:huge}
A curve with whiskers is thinly homotopic to a curve without.
\end{theorem}

\begin{proof}
We proceed in four steps:

\textit{Step 1: Any whisker factors through a tree. More precisely, there is a tree $T$ and two maps $\tilde{\gamma}: I \to T$ and $f: T \to M$, such that $\gamma = f \circ \tilde{\gamma}$ with $\tilde{\gamma}$ continuous and $f$ Lipschitz with constant 1.}  

Consider the word that corresponds to the curve. We know that every letter appears a finite number of times together with its inverse. Replace any letter which is repeated $n$ times with $n$ different letters, making sure to preserve the pairing between the letter and its inverse. Of course, with this every letter appears precisely twice (once itself and once as an inverse). Whenever needed, interchange the letter with its inverse so that the letter appears first. Of course, the resulting word is still reducible.

Consider now the following geometric arrangement:

Embed $I$ in the closed upper half plane in the natural way (on the $x$-axis). For any open interval in one of the sets $\{ A_n \}_{n=1}^{\infty}$, we know that there is another open interval (in the same $A_n$ in fact) such that if the first one corresponds to a letter then the second one corresponds to the inverse of the letter. Associate to every such pair the open set contained between the semicircle in the upper half plane based at the outer endpoints of the two intervals and the semicircle based at the inner endpoints. We define a partial order on these open semi-annuli by saying that the one before is the one `above'.\footnote{More precisely, we could say that the one before is the one which if combined with its reflection across the $x$-axis will contain the other one.} Figure \ref{fig:annuli} should make this clear. We will often use the word `above' instead of `before' or `preceds' and the word `below' instead of `later', `after' or `succeeds'.

\begin{figure}
\begin{center}
\includegraphics[width = 0.8\textwidth]{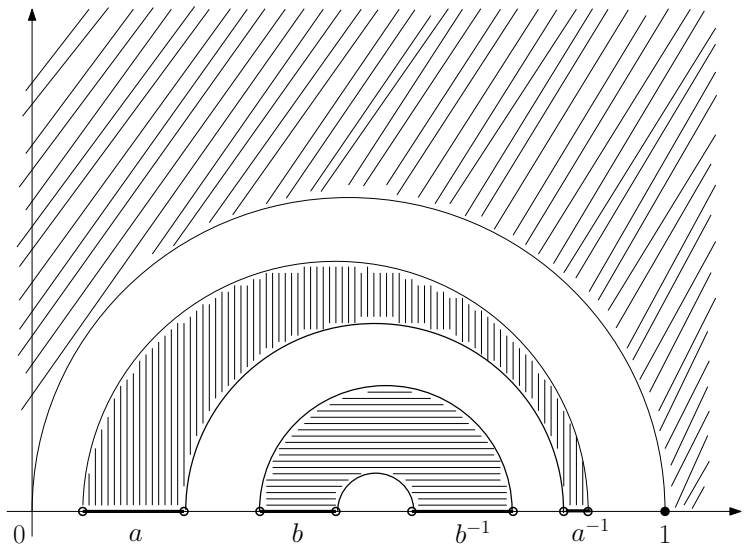}
\caption{Three semi-annuli are shown. If we denote the vertically shaded one by $x$, the horizontally shaded one by $y$ and the diagonally shaded region by $z$, then $z$ preceds $x$ which is succeeded by $y$. In fact, $z$ is the minimal semi-annulus. Note that all the three regions are open sets in the subspace topology of the closed upper half-plane.}
\label{fig:annuli}
\end{center}
\end{figure}

Perform this for every pair of intervals. Note that the sets obtained for different pairs do not intersect. Take now the set of all these semi-annuli as well as the set $\{(x,y) \in \mathbb{R}^2 : y \geq 0 \, , \, (x-\frac{1}{2} )^2 + y^2 > \frac{1}{4} \}$. Extend the partial order on the semi-annuli to include this set by making it the minimal element. We will also call this minimal element a semi-annulus. Let us analyze how a path connected component of the complement of this set would look like.

Let $V$ be such a component. Consider the set of all semi-annuli above it.\footnote{Here we are abusing the terminology slightly, since the partial order was only defined for semi-annuli, but the meaning of this statement should be clear.}  If there is a latest one, whose smaller boundary circle has equation $(x-x')^2 + y^2 = \rho^2$ with $(x',0)$ being its center and $\rho$ its radius, then $V$ can contain only points satisfying $(x-x')^2 + y^2 \leq \rho^2$. If, on the other hand, there is no latest one, consider the net of inner semi-circles of these semi-annuli (the index set of the net being the semi-annuli under consideration with their partial order). For each element in this net, consider the three nets which are the net of the left endpoints, right endpoints and radii of the semicircles. Since the first two are monotone, they converge and thus the third one converges too and together the limits of all three define a semi-circle. This is because if these three nets are $l_{\alpha}, r_{ \alpha}, \rho_{\alpha}$, then they satisfy $2 \rho_{\alpha} = r_{\alpha} - l_{\alpha}$, and so, their limits will also satisfy this relation. If $x' = \frac{1}{2} \lim_{\alpha} (r_{\alpha} + l_{\alpha})$ and $\rho = \lim_{\alpha} \rho_{\alpha}$, then it is easy to see that $V$ can only contain points satisfying $(x-x')^2 + y^2 \leq \rho^2$.  

Now, if there is no semi-annulus which is contained in $(x-x')^2 + y^2 \leq \rho^2$, then $V$ is in fact equal to this set (restricted to nonnegative $y$'s of course). For reasons which will be apparent later we shall call such $V$'s `tip' regions.

\begin{figure}
\begin{center}
\includegraphics[width=1\textwidth]{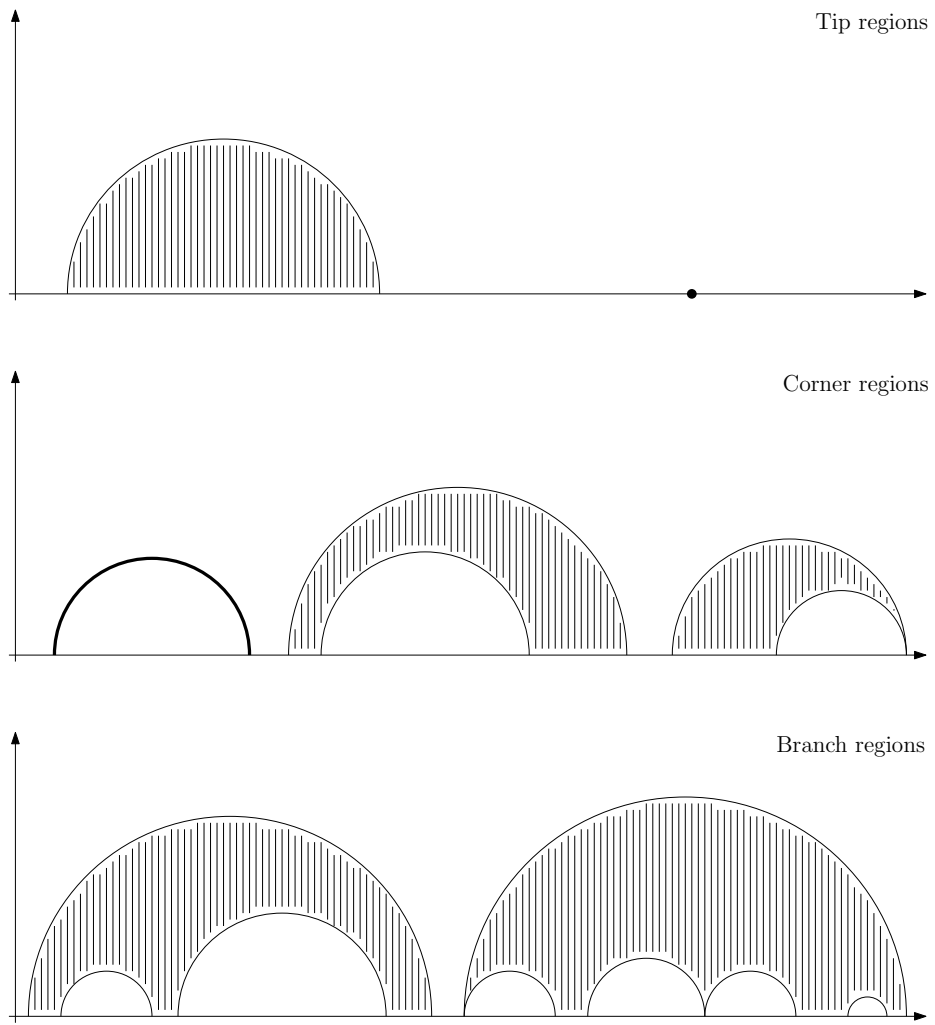}
\caption{This is a drawing of the different vertex regions. All the regions are closed and so the boundaries are assumed to be included. Note that a tip region can be just a point on the $x$-axis, while a corner region may be just a circle. The shown possibilities are not exhaustive.}
\label{fig:regions}
\end{center}
\end{figure}

If there is a semi-annulus which is contained in $\{ (x,y) : y \geq 0 \, , \, (x-x')^2 + y^2 \leq \rho^2\}$, consider all semi-annuli which precede it but which succeed all the semi-annuli preceding $V$. Performing the same analysis as above, either taking a minimal semi-annulus (with respect to the partial order) or taking a net with it being ordered oppositely to the semi-annuli, we can see that $V$ can only contain points which satisfy $(x-x_1)^2 + y^2 \geq \rho_1^2$ for some $x_1$ and $\rho_1$.

Now, if there are no semi-annuli contained in $\{ (x,y) : y \geq 0 \, , \, (x-x')^2 + y^2 \leq \rho^2 \}$ which do not succeed $(x-x_1)^2 + y^2 = \rho_1^2$ then of course $V$ is the intersection of $\{ (x,y) : y \geq 0 \, , \, (x-x')^2 + y^2 \leq \rho^2 \}$ and $\{ (x,y) : y \geq 0 \, , \, (x-x_1)^2 + y^2 \geq \rho_1^2 \}$. Such $V$'s will be called `corner' regions.

Finally, in the other case (i.e. when there are semi-annuli contained in $\{ (x,y) : y \geq 0 \, , \, (x-x')^2 + y^2 \leq \rho^2 \}$ etc.) repeat the procedure which led to $(x-x_1)^ 2 + y^2 = \rho_1^2$ starting from another semi-annulus and obtain another semicircle $(x-x_2)^2 + y^2 = \rho_2^2$ such that $V$ is constrained to lie above it. Note  that the new semi-circle is disjoint from the first one. Proceeding in this fashion, we arrive at a (possibly countable) collection of semicircles such that $V$ is the region above them. Such $V$'s will be called `branch' regions. Note that the set of branch regions is countable, due to the fact that each such region has a nonempty interior. This is a technical point which will be relevant later.

The reader should consult figure \ref{fig:regions}  for drawings of the different cases. We shall call the tip, corner and branch regions collectively `vertex' regions. We shall extend the partial order defined on the semi-annuli to the vertex regions in the obvious way as we have done above. Note that there is a minimal (in the set of vertices) vertex region: it is the region which is bounded from above by the minimal semi-annulus. We shall call this vertex region, the root region.

It is not difficult, though a little tedious, to see that, for each of the cases above, the curve takes all the points of $V$ which lie on the $x$-axis to the same image. This follows from the fact that the curve restricted to an interval in $A_n$ traverses the same image as when restricted to the interval which is paired with it, from the fact that the derivative of the curve vanishes in the interior of $A_0$ (the complement of $\bigcup_{n=1}^\infty A_n$) and from continuity (this is needed when taking limits of the nets).

Before we define our tree, we still need to make some further refinements. Consider any point in $A_0$. Then this point naturally splits the word associated to the curve into two parts, the one `before' and the one `after'. Note that both the words before and after are the same for all the points in $A_0$ which are in the same connected component, and thus we can talk about the words before and after a component of $A_0$. Given a word, we can talk about an initial/final segment of it, which is defined to be the restriction of the map defining the word to an initial/final segment of the linearly ordered set defining this word. Also, note that the intersection of any corner region with the $x$-axis always has two connected components. Using the above we can now distinguish two types of corner regions:

\begin{itemize}
\item `Spurious' : These are the ones which satisfy the following two conditions:
\begin{itemize}
\item Let us say the two components of this region on the $x$-axis are $C_1$ and $C_2$ with $C_1$ being to the left of $C_2$. Then there is a final segment of the word before $C_1$ which is paired with an initial segment of the word after $C_2$. Similarly, there is an initial segment of the word after $C_1$ which is paired with the final segment of the word before $C_2$.
\item The curve is $C^1$ if parametrized by proper length in the neighborhood of the images of both $C_1$ and $C_2$.\footnote{Proper length is defined of course by $l(t) = \int_0^t \sqrt{ \dot{\gamma}^2(s)} ds$. Note that a sufficient condition for the curve to be $C^1$ as a function of the proper length is the non vanishing of $\dot{\gamma}$, but that this is not necessary.}
\end{itemize}
\item `True' : All of the rest of the corner regions.
\end{itemize}

Having classified all the corner regions into the above two classes, we redefine the decomposition of the domain of the curve by fusing all those paired intervals which are interspersed by the (restriction to the $x$-axis of the) spurious corner regions. More precisely, if $w$ is paired with $w^{-1}$ and all the regions between the letters of $w$ (and consequently of $w^{-1}$) are of the spurious corner type, we replace $w$ by the smallest open interval containing all the intervals corresponding to all the letters of $w$ and perform the same for $w^{-1}$. We then replace the word $w$ with one letter and $w^{-1}$ with its inverse, keeping them paired together. We proceed recursively until there are no more spurious corner regions left. Thus since we only have true corner regions left, we shall simply call them corner regions from now on.

Now, let us define the tree. It shall be embedded as a subspace of $l^1$. Let the letters appearing in the word of the curve be given by the sequence $\{ a_n \}_{n=1}^\infty$, and let the proper length of the arcs corresponding to them be $\{l_n \}_{n=1}^\infty$ (we will often abuse notation and say that `$l_i$ is the length of $a_i$' or it is `the length of the semi-annulus corresponding to $a_i$'). To any vertex region, associate the point in $l^1$ whose $i$-th coordinate is $l_i$ if the semi-annulus corresponding to $a_i$ is above this region, and is 0 otherwise. The set of all such points will be the vertices of the tree (hence the name of these regions). The edges of the tree will be in one to one correspondence with the letters of the word. The edge corresponding to $a_i$ will stretch from the vertex corresponding to the vertex region which bounds the semi-annulus associated to $a_i$ from above to the one associated to the vertex region bounding it from below along the $i$-th axis. We shall denote the tree by $T$. Note that excluding the endpoints of any edge, the rest of the edge is an open subset of $T$, and thus the set of vertices is closed in $T$. It is not hard to see that the vertices are a totally disconnected set as, for any two vertex regions, there is a semi-annulus which is above one of these regions which is not above the other one. The names for the vertex regions discussed above should be self-explanatory now, as the tip regions correspond to the tips of the tree. The corner regions correspond to corners and the branch regions correspond to the points where the tree splits into several branches. We shall call the origin of $l^1$ the root of the tree, and of course, it corresponds to the root region. Intuitively, $T$ is just the Poincar$\acute{\txt{e}}$ dual of the collection of vertex regions and semi-annuli.

Let us show that $T$ is compact. Let $\{x_n\}_{n=1}^\infty$ be a sequence of elements in $T$. We need to show that it has a convergent subsequence in $T$. It is sufficient to handle the case when this sequence has no subsequence which is contained inside an edge of $T$, for then this subsequence will converge to an element in this edge as an edge is compact. Therefore, without loss of generality, assume that no two $x_n$'s are in the same edge or vertex. Define now a sequence $\{t_n \}_{n=1}^\infty$ of elements of $I$ by letting $t_n$ be the left endpoint of the interval corresponding to the letter corresponding to the edge in the case $x_n$ belongs to the interior of an edge, while in case $x_n$ is a vertex, let $t_n$ be any of the points in the corresponding vertex region which are in $I$. The sequence $\{ t_n \}_{n=1}^\infty$ has a convergent subsequence $\{ t_{n_k} \}_{k=1}^\infty$ such that $t_{n_k} \to t_0 \in A_0$. Let $x_0$ be the vertex corresponding to the vertex region to which $t_0$ belongs. We then have that $x_{n_k} \to x_0$. The reason is that if $s_1, s_2 \in I$ are in vertex regions, then the $l^1$ distance between the vertices corresponding to these two regions is less than or equal to $|\int_{s_1}^{s_2} \sqrt{\dot{\gamma}^2(s) } ds|$ as this integral is less than or equal to the sum of the lengths of all the semi-annuli which precede one of the regions without preceding the other one. Thus $T$ is compact. Note that essentially the same proof shows that any sequence of vertices has a subsequence converging to a vertex and thus the set of vertices is compact as well.

To show that $T$ is a tree we need to show that it is path-connected and simply connected. It is obvious that if for any vertex we have a continuous path connecting the root to it, then $T$ is connected. Pick a vertex, or equivalently, a vertex region. Then there is a subcollection of semi-annuli which are above it. Let $L$ be the sum of the $l_i$'s corresponding to the letters corresponding to these semi-annuli, and define the path $\sigma(s)$ from $[0,L]$ to $T$ in the following way:

Consider the set of all the vertices preceding the vertex in question. Note that this set of points is in fact a closed subset of $T$. This is because it is equal to the intersection of two closed sets: the set of all vertices with the hyperplane which has all coordinates equal to 0 for all the semi-annuli not preceding the vertex in question.\footnote{Any other vertex region is going to be either below the vertex in question, and thus separated from it by some semi-annulus and will have the $l^1$ coordinate corresponding to it nonzero; or it is going to be `next' to it (as are the bottom two regions in figure \ref{fig:regions}) in which case it is also not difficult to see that there will be a semi-annulus above it which is not above the first vertex.} Now, for every vertex region $v$ above the vertex in question, let $s_v$ be the sum of the lengths of the semi-annuli above $v$. We thus obtain a linearly ordered set of real numbers which is closed as the map $(\alpha_1, \alpha_2, \dots) \to \alpha_1 + \alpha_2 + \dots$ is a continuous map from $l^1$ to $\mathbb{R}$. Note that this set of vertices is in fact linearly ordered and that if $v_1$ precedes $v_2$ with no other vertices in between then there is an edge whose initial point is $v_1$ and final point $v_2$. These facts are easy consequences of the geometry of the vertex regions. Let $s \in [0,L]$. If there is a $v$ such that $s_v = s$, let $\sigma(s_v) = v$. Otherwise, $s$ is in the complement of the closed set and thus belongs to an interval $(s_{v_1}, s_{v_2})$ such that there are no $s_v$'s in it. In this case, let $\sigma$ take $s$ to the edge linking $v_1$ and $v_2$, going along it (in the direction of $v_2$) a distance $s- s_{v_1}$. 

From the construction of $\sigma$ it follows easily that $| \sigma(s_1) - \sigma(s_2) | = |s_1 - s_2|$, and thus $\sigma$ is (uniformly) continuous and $T$ is path connected. 

In step 3, we shall construct an explicit deformation retraction of $T$ to its root, which will imply that $T$ is simply connected.

Having defined our tree, let us show how the curve factorizes through it.  Let us define $\tilde{\gamma} : I \to T$ first. Note that any point $t \in I$ belongs either to a vertex region or to a semi-annulus. In the case it belongs to a vertex region let $\tilde{\gamma}(t)$ be the vertex corresponding to that region. In the other case, if $t$ is in the earlier interval of the intersection of the semi-annulus with $I$, denoted by $(t_1, t_2)$, let $\tilde{\gamma}(t)$ be the point on the edge corresponding to the semi-annulus whose distance from its initial vertex (i.e. the vertex corresponding to the vertex region which contains $\{t_1 \}$) is equal to $\int_{t_1}^t \sqrt{ \dot{\gamma}^2(s)} ds$. If $t$ is in the second interval, which we will also denote by $(t_1, t_2)$, $\tilde{\gamma}(t)$ is defined via the same formula but the distance is now measured from the final vertex (the one corresponding to $t_1$). It is not hard to see that $\tilde{\gamma}$ satisfies the inequality $| \tilde{\gamma}(t_1) - \tilde{\gamma}(t_2) | \leq \int_{t_1}^{t_2} \sqrt{  \dot{\gamma}^2(s)}  ds $ from which continuity of $\tilde{\gamma}$ follows. Note that we can say that $\tilde{\gamma}(t)$ is in fact $C^1$ when $t$ maps to an interior of an edge.\footnote{When $\tilde{\gamma}$ maps into an edge, it is essentially a function into $\mathbb{R}$ as only one component is changing, and thus we can talk about its differentiability.}

Let us now define the map $f : T \to M$. For any vertex $x$, let $f(x)$ be the image by $\gamma$ of the intersection of the vertex region associated to $v$ with the $x$-axis. Let $x$ be on an edge, and suppose the distance from $x$ to the initial vertex of the edge is equal to $l$. Suppose that the left interval forming intersection of the semi-annulus corresponding to this edge with the $x$-axis is $(t_1, t_2)$. Let $f(x) = \gamma(t(x))$ where $t(x)$ is such that $\int_{t_1}^{t(x)} \sqrt{ \dot{\gamma}^2(s)  } ds = l$. It is straightforward to check that indeed, $\gamma = f \circ \tilde{\gamma}$, and that $|f(x) - f(y) | \leq |x - y|$. 

\textit{Step 2: Any whisker is thinly homotopic to a whisker which factors through the same tree and which has a vanishing derivative at every point which is mapped to a vertex.} 

Let $\mathcal{V}$ be the preimage of the set of vertices under $\tilde{\gamma}$. Since $\tilde{\gamma}$ is continuous and the set of vertices is closed, we have that $\mathcal{V}$ is closed. It is not hard to see that for any corner region at least one of its components is in $C$ (the set of critical points, where we are using the same notation as in theorem \ref{theorem:words}. Recall here that we assume that we have gotten rid of the spurious corners). 

Now, define the maps $l: I \to [0,L]$ and $\hat{\gamma} :  [0,L] \to M$, via

\besn
l(t) & = & \int_0^t \sqrt{ \dot{\gamma}^2(s)} ds, \\
\hat{\gamma} & = & \gamma \circ l^{-1},
\eesn

where of course $L$ is the total length of the curve, and $\hat{\gamma}$ is just the curve parametrized by proper length. We are abusing the notation somewhat, since $l$ is not invertible, and thus should be read as a preimage, but this does not cause any problems since $\gamma$ is of course constant on the preimage of any point by $l$. The reader should note that $\hat{\gamma}$ is a Lipschitz map (with constant 1) which is in fact $C^1$ at every point which is not a critical value for $l$ (since there $l$ is invertible).

Consider an interval $(a,b)$ forming the compelment of $C$. It is clear that $\hat{\gamma}$ is a $C^1$ diffeomorphism if restricted to $(a+\delta, b- \delta)$ for any (sufficiently small) $\delta >0$. Since any corner point is a subset of $\gamma^{-1} \big ( \gamma(C) \big )$, it is clear that the set of the images of the corners which is contained in $\hat{\gamma}\big ((a+\delta, b-\delta)\big)$ has 1-dimensional Hausdorff measure zero (Sard's theorem again). Which in turn implies that $\hat{\gamma}^{-1}$ of this set has measure 0 in $[0,L]$. Since the interval and $\delta >0$ are arbitrary, it follows that the image of the set of all corner points under $l$ has measure 0.

It is not difficult to see that the points which map to tips are in $C$, for in a neighborhood of any such point the curve is at least two-to-one. Consider now the branch points. Recall that any branch region intersected with the $x$-axis has exactly two components. If either of these two components is more than a point, then that component is of course contained in $C$. It follows from this, and from the fact that the set of branch regions is countable, that the set of points which map to the branch points and which are not in $C$ is countable.

Therefore, we have that $l(\mathcal{V})$ is a closed set and is of measure 0. We now need the following 

\begin{lemma}
If $S \subset I$ is any countable set such that $S \cup \{0\} \cup \{1\}$ is closed then there is a $C^1$, monotone increasing, surjective function $\psi: I \to I$, whose set of critical values is precisely $S \cup \{0\} \cup \{1\}$. Moreover one can assume that $A = \txt{sup}_{x \in I} | \psi'(x) | \geq 1$ is independent of the set $S$.
\end{lemma}

\begin{proof}
It is a minor tweak of the construction presented in \cite{bates}. See the Appendix for details.
\end{proof}

It follows at once, by scaling and shifting, that the same can be said for any other interval $J$, and that the same bound is in fact independent not only of $S$,  but also of which interval is under consideration. Let us denote such an adjusted function by $\psi_{J,S}$.

Now, define the function $\phi : [0,L] \to [0,L]$ by:

\begin{displaymath}
\phi(s) = \left\{ \begin{array}{ll}
s & \textrm{If $s \in l(C)$.}\\ \\
\psi_{[a,b], l(\mathcal{V}) \cap [a,b] } & \textrm{If $t \in [a,b]$ where $(a,b)$ is one of the open intervals} \\ & \textrm{making up $[0,L] - l(C)$.}
\end{array} \right.
\end{displaymath}

In other words, the function $\phi$ fixes all the points in $l(C)$ and has a vanishing derivative at any point which corresponds to a branch point. It is obvious that $\phi$ is $C^1$ at any $s \notin l(C)$. Moreover, it is easy to see that $\phi$ is in fact Lipschitz (with constant $A$) for, let $0 \leq s_1 \leq s_2 \leq [0,L]$, then

\besn
\phi(s_2) - \phi(s_1)  & = & | \phi([s_1, s_2] )|   \\
& = &  | \phi(l(C) \cap [s_1, s_2] )|  + | \phi( [s_1,s_2] - l(C)   )   | \\
& = &  | l(C) \cap [s_1, s_2] | + | \phi(  [s_1, s_2] - l(C))| \\
& \leq & | l(C) \cap [s_1, s_2]| + A | [s_1, s_2] - l(C)|.\\
& \leq & A (  |l(C) \cap[s_1, s_2] | + | [s_1, s_2] - l(C) | ) = A ( s_2 - s_1). 
\eesn

where above $| \cdot |$ stands for the Lebesgue measure on $\mathbb{R}$.

Define now the homotopy $H : I \times I \to M$ via

\ben
H(t, r) = \hat{\gamma} \circ \rho \circ l,
\een

where $\rho (s , r) = (1-r) s + r \phi(s)$ is just the straight line homotopy between the identity and $\phi$. Let us show that $H$ is $C^1$. It is clear that at any point $t_0 \notin C$, the above function has continuous partial derivatives, since the three constituent functions are continuously differentiable at the corresponding points. Let $t_0 \in C$, then it follows that $H(t_0 + \delta t, r) - H(t_0, r) = o(\delta t)$, as this is an immediate consequence of the fact that $l$ has a vanishing derivative at any such point and from the fact that both $\rho$ and $\hat{\gamma}$ are Lipschitz. Also, at such points we trivially have $H(t_0, r+ \delta r) - H(t_0, r) = 0 = o (\delta r)$. Therefore, $H$ has both partials everywhere. Moreover they both vanish when $t \in C$ and are continuous when $t \notin C$. Now fix $t_0 \in C$. Note that 

\be
\label{eq:partial1}
\frac{\partial H}{\partial t} = \hat{\gamma}' \frac{\partial \rho}{ \partial s } l'   
\ee

at any point $t \notin C$. However, the first two multiplicands are bounded everywhere (the differentiated functions are Lipschitz) while $l'$ vanishes at every point in $C$ and is continuous everywhere. It follows that 

\ben
\lim_{t \to t_0 \, , \, t \notin C} \frac{\partial H}{ \partial t} = 0 ,
\een

uniformly in $r$, and since $\frac{ \partial H}{\partial t} = 0$ at any $t \in C$, we have that $\frac{\partial H}{\partial t}$ is continuous everywhere. Note that (\ref{eq:partial1}) implies that $|\frac{\partial H}{\partial t}| \leq A \, \,  \ \txt{sup}_{ t \in I} | l'(t)|$, where the constant $A$ was defined above. 

Consider now the other partial. If $t \notin C$, then this partial is given by

\be
\label{eq:partial2}
\frac{\partial H}{\partial r} = \hat{\gamma}' \big (  \phi(l(t)) - l(t)   \big ).
\ee

Note that if $t_0 \in C$, then $\phi( l(t)) - l(t) = (t- t_0) ( ( \phi\circ l)'(\tau) - l' (\tau) )$ for some $\tau$ between $t_0$ and $t$ by mean value theorem. Here we used the fact that $\phi \circ l$ is in fact differentiable, despite $\phi$ by itself being only Lipschitz as it is not differentiable precisely at the points where $l$ has a vanishing derivative. Using the fact that $\phi$ is Lipschitz and that it is differentiable everywhere away from $C$, it follows that $\phi(l(t)) - l(t) \to 0$ as $t \to t_0$. Utilizing the boundedness of $\hat{\gamma}'$ again, we get that $\frac{ \partial H}{\partial r}$ is continuous everywhere. Therefore $H$ is $C^1$. Also, (\ref{eq:partial2}) and mean value theorem again imply that $|\frac{\partial H}{\partial r}| \leq  A \, \,  \txt{sup}_{ t \in I} | l'(t)|$.  

The fact that $H(t,0) = \gamma(t)$ is obvious, and that $H(t,1)$ is a whisker which factors through the same tree is immediate from the monotonicity of $\phi$. It is clear that $H(I \times I) \subset \gamma(I)$, which forces $H$ to have rank at most 1. 

Finally, to get a thin homotopy we should have the appropriate derivatives vanish at the boundary of $I \times I$. It is obvious that the derivative vanishes at the two vertical edges forming $\partial (I \times I)$. If we let, e.g., $\tilde{H}(t, r) = H( t , \psi(r))$ we see that $\tilde{H}$ is the thin homotopy that we want.

\textit{Step 3: Any whisker is thinly homotopic to the constant map.}

The reader is urged at this point to read the Appendix as a closely related argument will be utilized in this step. 

It was proven in steps 1 and 2 that$\gamma$ factors through a tree (we haven't proven yet that it is simply connected), i.e. $\gamma = f \circ \tilde{\gamma}$, where $f$ is Lipschitz with constant 1, $\tilde{\gamma}$ is continuous everywhere, $C^1$ in the interior of every edge, and without loss of generality is $o(\delta t)$ when it is mapped to a vertex as we can assume that we have thinly homotoped the curve to one which has a vanishing derivative at any point mapping to a vertex.

Consider now the set of edges of the tree. Arrange them so that their lengths $\{ l_n \}_{n=1}^\infty$ are non-increasing. Assume that $\sum_{n=1}^\infty l_n = L''$. Associate to every edge an interval $l_n'$ such that $\sum_{n=1}^\infty l'_n = L' < \infty$, $l_n \leq l_n'$ and $\lim_{n \to \infty} \frac{l_n}{l_n'} = 0$. 

Recall that in step 1 we have constructed for any vertex $x$ a continuous path $\sigma: [0,|x|] \to T$, linking the root to this vertex.\ This path satisfied the property that $|\sigma(s) - \sigma(t) | = |s - t|$, and in particular $|\sigma(s)| = s$. Since every edge begins with a vertex, it is trivial to define an analogous map when $x$ is in the interior of any edge of $T$. Thus, for any $x \in T$, we have a path denoted by $\sigma_x: [0,|x|] \to T$, such that $\sigma_x (s)$ is the point on this path a distance $s$ from the root. It is obvious that $\sigma_x$ is Lipschitz with constant 1 and if $\sigma_x(s)$ is in the interior of an edge, then $\sigma_x$ is $C^1$ at such a point.

Note that if $x$ is a vertex, then $|x|$ is the sum of the lengths of all the edges traversed by $\sigma_x$. Define $|x|'$ to be the sum of the $l_n'$'s of the edges in this path. For any vertex $x$, denote by $\rho_x$ the function which takes $[0, |x|'] \to [0, |x|]$, which is $C^1$, monotone increasing, and when $s = |y|'$ for some vertex $y$ along this path, $\rho_x (s) = |y|$ and $\frac{d}{ds}\rho_x(s) = 0$. This is done by mapping $s = |y|'$ to $s = |y|$ for any vertex $y$ on this path and by interpolating by a (scaled and shifted) version of the same monotone, $C^1$ function with vanishing derivative at the endpoints. The proof that the resulting function is $C^1$ is a version of the proof in the Appendix.

The crucial fact about $\sigma_x$ and $\rho_x$ is that if $y$ is a vertex on the path $\sigma_x$ then $\rho_x = \rho_y$ and $\sigma_x = \sigma_y$ on the common domains as is evident from the construction of these functions. This fact guarantees the consistency of the definition of the map $\chi$ which shall be given momentarily.

Let $\chi : T \times L' \to T$ be defined in the following way:

\begin{itemize}
\item If $x$ is a vertex, let 

\begin{displaymath}
\chi(x, r) = \left \{ \begin{array}{ll}
x & \txt{if $r \leq L' - |x|'$,}\\
\sigma_x \circ \rho_x (L' - r) & \txt{if $r \geq L'- |x|'$}.\\
\end{array}
\right.
\end{displaymath}

\item If $x$ is a point in the interior of an edge $e$ whose initial vertex is $v_1$ and whose final vertex is $v_2$, and $| x - v_1| = \alpha |v_2 - v_1|$, let

\begin{displaymath}
\chi(x, r) = \left \{ \begin{array}{ll}
x & \txt{if $r \leq L' - |x_2|'$,}\\
x_1 + \alpha \, \chi (x_2, r) & \txt{if $ L' - |x_2|' \leq r \leq L'  - |x_1|'$}.\\
\chi(x_1, r) & \txt{if $ L' - |x_1|' \geq r$}.
\end{array}
\right.
\end{displaymath}

\end{itemize}

Note that the above map is in fact a continuous deformation retraction of $T$ to its root, and thus we have shown that $T$ is simply connected. 

Now, define the homotopy $H : I \times L' \to M$ via:

\ben
H = f \circ \chi \circ \tilde{\gamma}.
\een

Roughly speaking, $T$ is contracted within itself using $\chi$, such that the tips `stop' whenever they pass by a vertex. It is clear, intuitively, that $H$ is $C^1$ as the only troublesome points are the vertices and the holonomy `comes to a halt' at all such points, in both variables. At any rate, it is not hard, though quite tedious, to verify this rigorously. First, by considering the different cases, one can show $\frac{\partial H}{\partial t}$ and $\frac{\partial H}{\partial s}$ exist everywhere. There are five cases to consider for each partial, corresponding to the five possibilities in the definition of $\chi$. It is easy to see that both partials vanish whenever the image of a point under $\chi \circ \tilde{\gamma}$ is a vertex, and that they satisfy the inequalities $|\frac{\partial H (t,s)}{\partial t}| \leq | \dot{\gamma} (t) |$ and $|\frac{\partial H (t,s)}{\partial s}| \leq \rho'_x (L' - s) $ where $x$ is any vertex succeeding $\tilde{\gamma}(t)$. Using the fact that $\frac{\partial H (t_0, s)}{\partial t} =0$ if $t_0$ is in a vertex region and the first inequality, it is easy to see that this partial is continuous. Continuity of the other partial is a little more subtle: one needs to use the fact that $H(t_0, s)$ is a $C^1$ function as a function of $s$ if $t_0$ and then considering the various possibilities for the structure of the word around $t_0$, essentially, whether there are infinitely many letters in any neighbourhood of $t_0$ or not. 

It is trivial to check that $H(t, 0) = \gamma$ and $H(t, L')$ is the constant path, and it is obvious that $H$ has rank at most 1 as its image is contained in $\gamma(I)$. 

Finally, scaling and composing $H$ with a function which vanishes at the endpoints we obtain our thin homotopy.

\textit{Step 4: A curve with whiskers is thinly homotopic to one without.}

Assume the curve has whiskers, and consider the subwords which when reduced would be trivial. By concatenation and inclusion, we can assume that any such word is maximal (i.e. is not contained in a larger subword which is trivial when reduced). It follows that any two such words are separated by a nontrivial letter. Take the closure of the preimage of each such word. It should be clear that any two such preimages are disjoint (otherwise the two words could be concatenated and would form a larger reducible one). Each one of these preimages is a closed interval, but the set of all such intervals is not necessarily closed. However, since every tree contains at least one tip, and thus a point where the derivative vanishes, it follows that any limit point of these intervals which is not in one of them is in $C$. 

Now, for any whisker consider its root. Take its preimage under $\gamma$ and note that any interval in $I - C$ will contain at most two such points (belonging to two different whiskers, as there is always a point in $C$ in the `interior' of every whisker). Performing a similar procedure to the one in step 2 above (but simpler since $S$ now has at most two points), we can have a thin homotopy between the given curve to the one which has the derivative vanish at the root of every whisker.

Now, we can just apply the previous three steps to each one of the whiskers, performing all the homotopies simultaneously. The only subtle point is that one should take the edges of the entire curve  and not of the individual whiskers separately when defining the sequence $\{l'_n \}_{n=1}^\infty$ in step 3. Clearly, the two partials exist everywhere and are continuous away from the limit points of the set of roots of the whiskers. Using the two inequalities on partials at the end of the previous step it is easy to show that the partials are continuous at such limit points as well. Thus the resulting `total' homotopy is $C^1$. And since the images of those homotopies are always contained in the image of the curve the total homotopy is of rank at most one.

The fact that the resulting curve has no whiskers is obvious and the proof is complete. \end{proof}

We can now prove the following

\begin{theorem}
Assume $\gamma : I \to M$ is a loop with a vanishing derivative at the endpoints. Then the following are equivalent:
\begin{itemize}
\item[a-] $\gamma$ is a whisker.
\item[b-] $\gamma$ is thinly homotopic to the constant loop.
\item[c-] If $G$ is a semi-simple Lie group, then the holonomy of any smooth $G$-connection around $\gamma$ is trivial.
\item[d-] $\gamma$ factors through a tree via continuous maps.
\end{itemize}
\end{theorem}

\begin{proof}
$(a) \implies (b)$ is nothing but a special case of theorem \ref{theorem:huge} above.

$(b) \implies (c)$ was proven in \cite{caetano}.

To see that $(c) \implies (a)$, assume $\gamma$ is not a whisker. Thus, its reduced word is nontrivial. Therefore, one of the truncations to finitely many letters is not reducible. Let us denote the letters in this truncation by $a_1, \dots, a_n$ and the corresponding word by $\omega$. Consider the images of the intervals corresponding to these letters. They are, when restricted to their interiors, disjoint embeddings. It follows at once that given any Lie group $G$, and an $n$-tuple of group elements $g_1, \dots, g_n$, that there is a smooth $G$-connection whose parallel transport along the loop is the image of $(g_1, \dots, g_n)$ under the word map $G^n \to G$ defined by the (truncated) word.  However, it was shown in \cite{borel}, that the image of a word map $G^n \to G$ defined by a nontrivial word, would contain a nonempty open subset of the identity if $G$ is compact and semi-simple. In particular it would contain nontrivial elements. It is not difficult to see that even if $G$ is not compact the image of the word map will be nontrivial. The reason for this is that since $G$ is semi-simple then, if $K$ is its analytic subgroup (from the Cartan decomposition) and $Z$ is its center, then $K/Z$ is a compact semi-simple group. It is clear that the word map descends to the quotient $(K/Z)^n$. The image of this descended map is non-trivial, and thus its image before quotienting must have been non-trivial as well. And thus, there is a smooth $G$-connection whose holonomy around $\gamma$ is nontrivial.

$(a) \implies (d)$ is steps 1 and 3 of theorem \ref{theorem:huge}. It remains to prove that $(d) \implies (a)$. Suppose that the curve factorizes through a tree. We shall shift the tree to put its root at the origin. We will be done if we can show that the word associated to the curve is reducible.  Suppose $\gamma = f \circ \tilde{\gamma}$ with $\tilde{\gamma}:I \to T$. It is easy to see that if $a$ and $b$ are two different letters in the word, then their images (i.e. the images of the open intervals corresponding to them) under $\tilde{\gamma}$ are disjoint. Additionally, it follows, using theorem 1 and simple connectivity of the tree, that the images of the same letter under $\tilde{\gamma}$ are either identical or disjoint. We shall treat different occurrences of the same letter which have different images under $\tilde{\gamma}$ as different letters. For any letter, pick an edge whose interior is contained in the image of the letter under $\tilde{\gamma}$. Due to the way we have assumed our trees to be embedded in $l^1$, it follows that $|\tilde{\gamma}(t)|$ is a monotone function of $t$ when restricted to map to an edge. Call a letter positive if this function is increasing and negative if it is decreasing. It is easy to see that any letter must appear an even number of times, being alternatingly positive and negative, with the first occurrence in the word being positive. Pair a positive occurence of the letter with the first next negative one. It is straightforward to check that this pairing reduces the word to a trivial one and the proof is complete.
\end{proof}

We now have the following immediate corollary:

\textbf{Corollary : } If $\gamma_1$ and $\gamma_2$ are two curves with coinciding endpoints at which their derivatives vanish then the following are equivalent:

\begin{itemize}
\item[a-] The reduced word associated to $\gamma_1$ is the same as that for $\gamma_2$.
\item[b-] $\gamma_1$ is thinly homotopic to $\gamma_2$.
\item[c-] $\gamma_1$ and $\gamma_2$ are holonomically equivalent for a semi-simple $G$.
\item[d-] $\gamma_1\cdot \overline{\gamma}_2$ factors through a tree via continuous maps.
\end{itemize}

\section*{Appendix}

We sketch here the proof of the modification of the results of \cite{bates} that was used above.

\begin{lemma}
Assume $S \subset I$ is a closed set, of measure 0 which contains both $\{0\}$ and $\{1\}$. Then there is a monotone increasing surjective function $\psi : I \to I$ whose set of critical values is precisely $S$. Moreover, one can assume that $A = \txt{sup}_{x \in I} |\psi'(x) | \geq 1$ is independent of $S$.
\end{lemma}

\begin{proof}
The complement of $S$ is a countable collection of disjoint open intervals. Note that it is naturally a linearly ordered set which we shall denote by $\mathcal{I}$. Assume the lengths of these intervals are $\{l_n \}_{n=1}^\infty$, and have been ordered to be non-increasing. Since $|S|=0$, if follows that $\sum_{n=1}^\infty l_n = 1$. It is easy to see that there is a non-increasing sequence $\{ l'_n \}_{n=1}^\infty$ which satisfies:

\begin{itemize}
\item $l'_n \geq l_n$ for all $n \in \mathbb{N}$. 
\item $\sum_{n=1}^\infty l'_n = 2$.
\item $\lim_{n \to \infty} \frac{l_n}{l'_n} = 0$.
\end{itemize}

Let $f : I \to I$ be a $C^1$, monotone increasing, onto function whose derivative vanishes only at $\{0 \}$ and $\{ 1 \}$, such that $\txt{sup}_{x \in I } |f'(x) | = 2$.

Suppose $x \in S$, let $g(x) = \sum_{i \in \mathcal{I}, i < x} l'_i$. It is easy to see that $g: S \to [0,2]$ is a continuous, monotone injection. Therefore $S' = g(S)$ is a closed set. Thus, its complement is a countable collection of open intervals. Suppose, $(y_1, y_2) = (g(x_1), g(x_2))$ is one such interval. Note that if $x_2 - x_1 = l_n$ then $y_2 - y_1 = l'_n$. It follows that we have a natural bijection between the elements of $\mathcal{I}$ and the complement of $S'$, such that an interval of length $l_n$ corresponds to an interval of length $l'_n$. 

Define a function $\widehat{\psi}$ to be simply an extension of $g^{-1}$ from $S'$ to all of $[0,2]$ using a scaled, translated version of $f$ to map the complement of $S'$ onto the complement of $S$. More precisely, if $a'$ is the endpoint of an interval forming the complement of $S'$, and $a$ is the endpoint of the corresponding interval forming the complement of $S$, then $\widehat{\psi}$ restricted to $[a', a'+l']$ is given by:

$$
\widehat{\psi}(x) = a + l f \Big (   \frac{x- a'}{l'} \Big )
$$

It is obvious that $\widehat{\psi}$ is monotone, that it is $C^1$ on the complement of the set of limit points of $S'$, that it is differentiable everywhere, and that its set of critical values is precisely $S$. That $\widehat{\psi}$ is $C^1$ everywhere follows easily from the fact that $\lim_{n \to \infty} \frac{l_n}{l'_n} = 0$. Finally, it is trivial to verify that $\txt{sup}_{x \in [0,2]} |\widehat{\psi}'(x) |  \leq 2$. Setting $\psi(x) = \widehat{\psi}( \frac{x}{2} )$ we obtain the function we want. \end{proof}

\ni \texttt{{\footnotesize Department of Mathematics, American University of Beirut, Beirut, Lebanon.}
}
\texttt{\footnotesize{Email}} : \textbf{\footnotesize{tamer.tlas@aub.edu.lb}} 

\end{document}